\documentclass[12pt, a4paper]{amsart}

\usepackage[hmargin=30mm, vmargin=25mm, includefoot, twoside]{geometry}
\usepackage[bookmarksopen=true]{hyperref}

\usepackage{amscd}
\usepackage{amsfonts,amssymb,verbatim}
\usepackage{latexsym}
\usepackage{mathrsfs}
\usepackage{stmaryrd}
\usepackage{xspace}
\usepackage{enumerate, paralist}
\usepackage{graphicx}
\usepackage[all]{xy}
\usepackage{extarrows}
\usepackage{array}

\usepackage[usenames,dvipsnames]{color}

\usepackage{txfonts, pxfonts}

\usepackage{amsthm}
\usepackage{amsmath}

\usepackage{todonotes}
\newtheorem{thm}{Theorem}[section]
 \newtheorem{cor}[thm]{Corollary}
 \newtheorem{lem}[thm]{Lemma}
 \newtheorem{prop}[thm]{Proposition}

\newtheorem{alphthm}{Theorem}			

\newtheorem{alphcor}[alphthm]{Corollary}

 \theoremstyle{definition}
  \newtheorem{defn}[thm]{Definition}
  
  \newtheorem{question}[thm]{Question}

 \theoremstyle{remark}
 \newtheorem{rem}[thm]{Remark}
  \newtheorem{ex}[thm]{Example}

\newtheorem*{claim*}{Claim}

\def\NN{\mathbb{N}}

\def\CC{\mathbb{C}}

\def\U{\mathcal{U}}
\def\V{\mathcal{V}}

\def\Nd{\mathcal{N}}
\def\F{\mathcal{F}}
\def\M{\mathcal{M}}

\def\B{\mathfrak{B}}

\def\Fin{\mathrm{Fin}}
\def\Ord{\mathrm{Ord}}
\def\diam{\mathrm{diam}}
\def\supp{\mathrm{supp}}
\def\nudim{\mathrm{nudim}}
\def\trnudim{\mathrm{trnudim}}
\def\asdim{\mathrm{asdim}}
\def\ppg{\mathrm{prop}}
\def\trasdim{\mathrm{trasdim}}
\def\Id{\mathrm{Id}}
\def\dim{\mathrm{dim}}

\def\trdim{\mathrm{trdim}}

\begin{document}

\title{Transfinite Extension of Nuclear Dimension}

\author{Jingming Zhu and Jiawen Zhang}

\address[J. Zhu]{College of Data Science, Jiaxing University, Jiaxing, 314001, P.R.China.}
\email{jingmingzhu@zjxu.edu.cn}

\address[J. Zhang]{School of Mathematical Sciences, Fudan University, 220 Handan Road, Shanghai, 200433, China.}
\email{jiawenzhang@fudan.edu.cn}

\date{}

\thanks{J. Zhu was supported by the National Natural Science Foundation of China under Grant No.12071183. J. Zhang was supported by National Key R{\&}D Program of China 2022YFA100700.}

\begin{abstract}
In this paper, we introduce a notion of transfinite nuclear dimension for $C^*$-algebras, which coincides with the nuclear dimension when taking values in natural numbers. We use it to characterise a stronger form of having nuclear dimension at most $\omega$ and moreover, we show that the transfinite nuclear dimension of a uniform Roe algebra is bounded by the transfinite asymptotic dimension of the underlying space. Hence we obtain that the uniform Roe algebra for spaces with asymptotic property C has the corona factorisation property.
\end{abstract}

\date{\today}
\maketitle

\parskip 4pt


\textit{Keywords: Nuclear dimension, Transfinite nuclear dimension, (Transfinite) asymptotic dimension, Corona factorisation property.}

\section{Introduction}\label{sec:intro}

The notion of nuclear dimension was introduced by Winter and Zacharias in \cite{WZ10} as a non-commutative analogue of the topological covering dimension for nuclear $C^*$-algebras. It plays an important role in the classification of $C^*$-algebras, as the class of simple separable unital $C^*$-algebras which have finite nuclear dimension and satisfy the Universal Coefficient Theorem (UCT) can be classified by the Elliott invariant (see \cite{TWW17} and the references therein). For simple nuclear $C^*$-algebras, the Toms-Winter conjecture (\cite{ET08, Win10}) predicts that having finite nuclear dimension is equivalent to $\mathscr{Z}$-stability, which was proved in \cite{CETWW21, Win12} and has striking applications.

An important class of non-commutative $C^*$-algebras was introduced by J. Roe in his pioneering work \cite{Roe88, Roe96} on higher index theory. More precisely, given a discrete metric space $(X,d)$ of bounded geometry we can construct a $C^*$-algebra $C^*_u(X)$ (see Section \ref{sec:unif. Roe} for precise definitions), called the uniform Roe algebra of $X$, whose $K$-theory encodes the information of higher indices. It turns out (see \cite{STY02}) that the uniform Roe algebra $C^*_u(X)$ being nuclear is equivalent to the underlying space $X$ having property A in the sense of Yu (\cite{Yu00}).

Moreover, it was proved in \cite[Theorem 8.5]{WZ10} that the nuclear dimension of the uniform Roe algebra $C^*_u(X)$ is bounded by the asymptotic dimension of the space $X$. Here the notion of asymptotic dimension was introduced by Gromov in \cite{Gro93} as a coarse analogue of the topological covering dimension, which plays a key role in \cite{Yu98} to attack the coarse Baum-Connes conjecture.

Later Radul extended the codomain of the asymptotic dimension from natural numbers to ordinal numbers, and introduced a transfinite extension for the asymptotic dimension (``trasdim'' for abbreviation) in \cite{Rad10}. This transfinite extension classifies the metric spaces with the asymptotic property C introduced by Dranishnikov \cite{Dra00} (see Section \ref{sec:unif. Roe} for a precise definition) and attracts more and more attentions recently \cite{WZ21, WZR22, ZW20}.

In this paper, we study a transfinite extension for nuclear dimension and introduce a notion of transfinite nuclear dimension (``trnudim'' for abbreviation) for $C^*$-algebras. More precisely for each ordinal number $\alpha$, we define ``$\trnudim~A \leq \alpha$'' for a $C^*$-algebra $A$ (see Definition \ref{defn:trnudim}). We show that this indeed generalises the original notion of nuclear dimension as follows:

\begin{alphthm}\label{introthm:trnudim and nudim}
For a $C^*$-algebra $A$ and $n\in \NN$, we have $\trnudim~A \leq n$ \emph{if and only if} $\nudim~A \leq n$.
\end{alphthm}

On the other hand, Robert introduced a notion called ``having nuclear dimension at most $\omega$'' in \cite{Rob11} and studied the $\omega$-comparison property and the corona factorisation property from \cite{OPR12}. We consider a stronger form called ``strongly having nuclear dimension at most $\omega$'' (see Definition \ref{defn:nudim at most omega stronger form} for a precise definition) and provide the following characterisation to make a bridge with our notion of transfinite nuclear dimension:

\begin{alphthm}\label{introthm:omega nudim and trnudim}
For a $C^*$-algebra $A$, we have $\trnudim~A < \infty$ \emph{if and only if} $A$ strongly has nuclear dimension at most $\omega$.
\end{alphthm}

Combining Theorem \ref{introthm:omega nudim and trnudim} with \cite[Theorem 3.4 and Corollary 3.5]{Rob11}, we reach the following:

\begin{alphcor}\label{introcor:cor to thm omega}
Let $A$ be a $C^*$-algebra with $\trnudim~A < \infty$, then $A$ is nuclear and the Cuntz semigroup $\mathrm{Cu}(A)$ has the $\omega$-comparison property (see \cite[Definition 2.11]{OPR12}). Moreover if $A$ is $\sigma$-unital, then $A$ has the corona
factorisation property.
\end{alphcor}

As mentioned above, Winter and Zacharias showed in \cite[Theorem 8.5]{WZ10} that the nuclear dimension of the uniform Roe algebra $C^*_u(X)$ for a discrete metric space $(X,d)$ of bounded geometry is bounded by the asymptotic dimension of $X$. We prove the following transfinite generalisation of their result:

\begin{alphthm}\label{introthm:trnudim and trasdim}
Let $(X,d)$ be a discrete metric space of bounded geometry. Then we have $\trnudim~C^*_u(X) \leq \trasdim~X$.
\end{alphthm}

Combining with Radul's observation on asymptotic property C, we reach the following:

\begin{alphcor}\label{introcor:cor to thm asdim}
Let $(X, d)$ be a discrete metric space of bounded geometry with asymptotic property C, then the uniform Roe algebra $C^*_u(X)$ has the corona factorisation property and its Cuntz semigroup has the $\omega$-comparison property.
\end{alphcor}

The paper is organised as follows. In Section \ref{sec:pre}, we recall basic tools of ordinal numbers and some notions from metric/coarse geometry. In Section \ref{sec:trnudim}, we introduce our notion of transfinite nuclear dimension, prove Theorem \ref{introthm:trnudim and nudim} and Theorem \ref{introthm:omega nudim and trnudim} and provide an example. In Section \ref{sec:unif. Roe}, we recall the notion of uniform Roe algebras and (transfinite) asymptotic dimension, and prove Theorem \ref{introthm:trnudim and trasdim}. Finally in Section 5, we discuss the transfinite nuclear dimension in the commutative case.

\subsection*{Acknowledgements}
The first named author would like to thank Prof. Vladimir Manuilov for several helpful discussions. The second named author would like to thank Prof. Kang Li, Dr. Yanlin Liu and Dr. Jianguo Zhang for several useful comments after reading an early draft.

\section{Preliminaries}\label{sec:pre}

\subsection{The ordinal number Ord}

Here we recall the notion of ordinal number defined for collections of finite subsets from \cite{Bor88}.

Let $L$ be an arbitrary set and $\Fin~L$ be the collection of all finite and non-empty subsets of $L$. Let $M \subseteq \Fin~L$. For $\sigma \in \{\varnothing\}\cup \Fin~L$, denote
\[
M^{\sigma} = \{\tau\in \Fin~L: \tau \cup \sigma \in M \text{ and } \tau \cap \sigma = \varnothing\}.
\]
Let $M^a$ abbreviate $M^{\{a\}}$ for $a \in L$.

\begin{defn}[{\cite[Definition 2.1.1]{Bor88}}]\label{defn:ordinal number}
Let $L$ be an arbitrary set and $M \subseteq \Fin~L$. Define the \emph{ordinal number} $\Ord~M$ inductively as follows:
\[
\begin{array}{lll}
\Ord~M = 0 &\Leftrightarrow & M = \varnothing,\\
\Ord~M \leq \alpha &\Leftrightarrow & \forall~ a\in L, \Ord~M^a < \alpha,\\
\Ord~M = \alpha &\Leftrightarrow & \Ord~M \leq \alpha \text{ and } \Ord~M < \alpha \text{ is not true},\\
\Ord~M < \infty &\Leftrightarrow & \Ord~M \leq \alpha \text{ for some ordinal number } \alpha,\\
\Ord~M = \infty &\Leftrightarrow & \Ord~M \leq\alpha \text{ is not true for every ordinal number } \alpha.
\end{array}
\]
\end{defn}

We have the following easy observation:

\begin{lem}\label{lem:cal for Ord M}
Let $L$ be a set and $M \subseteq \Fin~L$. Set $\xi:=\sup\{\Ord~M^a: a\in L\}$. Then either $\Ord~M = \xi$ or $\Ord~M = \xi+1$.
\end{lem}

\begin{proof}
If there exists $a\in L$ such that $\Ord~M^a = \xi$, then by definition we have $\Ord~M = \xi+1$. Otherwise, by definition again $\Ord~M = \xi$.
\end{proof}

We record several useful lemmas from \cite{Bor88}.

\begin{lem}[{\cite[Lemma 2.1.2]{Bor88}}]\label{lem:ordinal inclusion lemma}
Let $L$ be a set, and $M \subseteq N \subseteq \Fin~L$. Then we have $\Ord~M \leq \Ord~N$.
\end{lem}

For a finite set $A$, denote $|A|$ the cardinality of $A$.

\begin{lem}[{\cite[Lemma 2.1.4]{Bor88}}]\label{lem:ordinal finite lemma}
Let $L$ be a set, $M \subseteq \Fin~L$ and $n\in \NN$. Then $\Ord~M \leq n$ \emph{if and only if} $|\sigma| \leq n$ for any $\sigma \in M$.
\end{lem}

\begin{lem}[{\cite[Lemma 2.1.6]{Bor88}}]\label{lem:ordinal bijection lemma}
Let $\Phi:L \to L'$ be a map from a set $L$ to a set $L'$. Let $M \subseteq \Fin~L$ and $M' \subseteq \Fin~L'$ such that for every $\sigma \in M$, we have $\Phi(\sigma) \in M'$ and $|\Phi(\sigma)| = |\sigma|$. Then $\Ord~M \leq \Ord~M'$.
\end{lem}

For a set $L$, a subset $M \subseteq \Fin~L$ is called \emph{inclusive} if for every $\sigma, \sigma' \in \Fin~L$ with $\sigma \in M$ and $\sigma' \subseteq \sigma$, then $\sigma' \in M$.

\begin{lem}[{\cite[Lemma 2.1.3]{Bor88}}]\label{lem:ordinal infinite lemma}
For a set $L$ and an inclusive subset $M \subseteq \Fin~L$, we have $\Ord~M = \infty$ \emph{if and only if} there exists a sequence $\{a_i\}_{i\in \NN}$ of distinct elements of $L$ such that $\sigma_n:=\{a_i\}_{i=1}^n \in M$ for each $n\in \NN$.
\end{lem}

The following results essentially come from \cite[Lemma 5 and Theorem 4]{Rad10}, where the special case of transfinite asymptotic dimension was proved. We can use exactly the same argument to prove the following. For the convenience to readers, we give the proofs.

\begin{lem}\label{lem:ordinal hereditary property}
Let $L$ be a set, $M \subseteq \Fin~L$ and $\tau \in \Fin~L \cup \{\varnothing\}$ such that $\Ord~M^\tau = \alpha$ for some ordinal number $\alpha$. The for each $\xi \leq \alpha$, there exists $\sigma \in \Fin~L \cup \{\varnothing\}$ such that $\Ord~M^{\tau \cup \sigma} = \xi$.
\end{lem}

\begin{proof}
We apply the transfinite induction on $\alpha$. If $\alpha=0$, the result holds trivially. Assume that the result holds for all $\alpha < \alpha_0$. We consider $M \subseteq \Fin~L$ and $\tau \in \Fin~L \cup \{\varnothing\}$ such that $\Ord~M^\tau = \alpha_0$, and let $\xi< \alpha_0$ be an ordinal number. Suppose that there is no $\sigma \in \Fin~L \cup \{\varnothing\}$ such that $\Ord~M^{\tau \cup \sigma} = \xi$. Then by the inductive assumption, there is no $\sigma' \in \Fin~L \cup \{\varnothing\}$ such that $\xi \leq \Ord~M^{\tau \cup \sigma'} < \alpha_0$. Then $\Ord~M^{\tau \cup \{a\}} < \xi$ for each $a\in L \setminus \tau$. This implies that $\Ord~M^\tau \leq \xi < \alpha_0$, which leads to a contradiction.
\end{proof}

\begin{prop}\label{prop:ordinal finite imply countable}
Let $L$ be a countable set and $M \subseteq \Fin~L$. If $\Ord~M< \infty$, then $\Ord~M < \omega_1$. Here $\omega_1$ is the first uncountable ordinal number.
\end{prop}

\begin{proof}
Suppose the contrary, \emph{i.e.}, $\Ord~M \geq \omega_1$. By Lemma \ref{lem:ordinal hereditary property}, we can choose $\tau \in \Fin~L \cup \{\varnothing\}$ such that $\Ord~M^\tau = \omega_1$. Hence for each $a\in L \setminus \tau$, we have $\Ord~M^{\tau \cup \{a\}} = \xi_a < \omega_1$. From Lemma \ref{lem:cal for Ord M} and the fact that $\omega_1$ is \emph{not} a successor ordinal, we obtain $\omega_1=\sup\{\xi_a: a\in L \setminus \tau\}$. Since each $\xi_a$ is countable and $L$ is also countable, we obtain a contradiction.
\end{proof}

\subsection{Notions from metric geometry}

Here we collect necessary notions and tools from metric geometry and coarse geometry (see, \emph{e.g.}, \cite{NY12}).

Let $(X, d)$ be a metric space and $U,V\subseteq X$. Denote
\[
\diam U=\sup\{d(x,y): x,y\in U\}
\quad \text{and} \quad
d(U,V)=\inf\{d(x,y): x\in U,y\in V\}.
\]
For $R>0$ and a family $\U$ of subsets of $X$, we say that $\U$ is a \emph{cover (of $X$)} if $X = \bigcup \U$. We say that $\U$ is \emph{$R$-bounded} if
\[
\diam \U:= \sup \left\{\diam  U: U\in \U\right\}\leq R,
\]
and $\U$ is \emph{uniformly bounded} if it is $R$-bounded for some $R>0$.
For $r>0$, we say that $\U$ is \emph{$r$-disjoint} if $d(U,V)\geq r$ for every $U,V\in \U$ with $U\neq V$. We say that $\U$ is \emph{strictly disjoint} if $\U$ is \emph{$r$-disjoint} for some $r>0$.

A discrete metric space $(X,d)$ is said to have \emph{bounded geometry} if for any $r>0$, the number $\sup\{|B_r(x)|:x\in X\}$ is finite, where $B_r(x):=\{y\in X:d(x,y)\leq r\}$.
For $A \subseteq X$ and $R>0$, denote $\Nd_R(A):=\{x\in X: d(x,A) < R\}$ the \emph{$R$-neighbourhood} of $A$.

For a function $f$ on a discrete metric space $(X,d)$, denote its \emph{support} by
\[
\supp(f):=\{x\in X: f(x) \neq 0\}.
\]
Given $\alpha>0$, say that a function $f$ on $(X,d)$ is \emph{$\alpha$-H\"{o}lder} if there exists $C>0$ such that $|f(x) - f(y)| \leq Cd(x,y)^\alpha$ for any $x,y\in X$. The \emph{$\alpha$-H\"{o}lder constant} of $f$ is defined to be the minimal $C$ satisfying the above inequality. When $\alpha=1$, then $f$ is called \emph{Lipschitz} and in this case, the minimal $C$ is also called the \emph{Lipschitz constant of $f$}.


We prove the following result for later use.


\begin{lem}\label{lem:cover for C(q)}
Let $(X,d)$ be a discrete metric space and  $q \in \NN$. For any $t_0 < t_1< \cdots < t_n$ in $\{3\cdot q^2 \cdot 2^{2k}: k\in \NN\setminus \{0\}\}$ and $t_i$-disjoint uniformly bounded family $\U_i$ of subsets of $X$ for $0\leq i \leq n$, there exist non-negative functions $f_0, f_1, \cdots,f_n$ on $X$ satisfying the following:
\begin{enumerate}
  \item $\supp(f_i) \subset \Nd_{t_i/2}(\bigcup \U_i)$ for $i=0,1,\cdots,n$;\\[-.3cm]
  \item $\sum_{i=0}^n f_i^2 \equiv 1$ on $\bigcup_{i=0}^n (\bigcup  \U_i)$; \\[-.3cm]
  \item each $f_i$ is $1/2$-H\"{o}lder with constant $C_i \leq \sqrt{3/t_i}$, and hence $\sum_{i=0}^n C_i < 1/q$.
\end{enumerate}
\end{lem}

\begin{proof}
For each $i=0,1,\cdots,n$, we firstly construct a function $g_i$ on $X$ as follows:
\[
g_i(x):=\max\left\{ 1-\frac{d(x,\bigcup\U_i)}{t_i/2}, 0\right\}.
\]
Now we construct $f_0, f_1, \cdots, f_n$ inductively. Firstly, we set $f_n:=g_n^{1/2}$. Note that the Lipschitz constant of $g_n$ is bounded by $2/t_n$ and hence using the inequality $|a^{1/2} - b^{1/2}| \leq |a-b|^{1/2}$ for $a,b \geq 0$, we obtain that $f_n$ is $1/2$-H\"{o}lder with constant $C_n \leq \sqrt{2/t_n}$. Now set $f_{n-1}(x):=\left(\max\{g_n(x),g_{n-1}(x)\} - g_{n}(x)\right)^{1/2}$ for $x\in X$. Similarly, we have $C_{n-1} \leq \sqrt{2/t_{n-1} + 2/t_n} \leq \sqrt{3/t_{n-1}}$. Inductively for each $i=0,1,\cdots,n-1$, we set:
\[
f_{i}(x)=\left(\max\{g_n(x),g_{n-1}(x),...,g_{i}(x)\}-\max\{g_n(x),g_{n-1}(x),...,g_{i+1}(x)\}\right)^{1/2} \quad \text{for} \quad x\in X.
\]
Similarly, we have $C_i \leq \sqrt{2/t_i + 2/t_{i+1}} \leq \sqrt{3/t_i}$. Note that $\sqrt{t_0}, \sqrt{t_1}, \cdots, \sqrt{t_n}$ are distinct elements in $\{\sqrt{3} \cdot q\cdot 2^{k}: k\in \NN\}$, and hence we have
\[
\sum_{i=0}^n C_i \leq \sum_{i=0}^n \sqrt{\frac{3}{t_i}} <\frac{1}{q},
\]
which concludes the proof.
\end{proof}

\section{Transfinite nuclear dimension}\label{sec:trnudim}

In this section, we introduce our notion of transfinite nuclear dimension for $C^*$-algebras, and prove Theorem \ref{introthm:trnudim and nudim} and Theorem \ref{introthm:omega nudim and trnudim}.

A completely positive (abbreviated as ``c.p.'') map $\phi:A \to B$ between $C^*$-algebras is said to have \emph{order zero} if $\phi(a)\perp \phi(b)$ for any $a,b \in A_+$ with $a \perp b$. Here $a \perp b$ means $ab=ba=0$. Also abbreviate ``completely positive contractive'' as ``c.p.c.''.

We record the following known fact:

\begin{lem}\label{lem:easy fact for order 0 map}
Let $\phi:A \to B$ be a c.p. order zero map between $C^*$-algebras. Then $\|\phi(a) \phi(b)\| \leq \|\phi\|^2 \cdot \|ab\|$ for any $a,b\in A$.
\end{lem}

\begin{proof}
By the theorem in \cite[Section 3]{WZ09}, there is a positive element $h \in \mathcal{M}(C) \cap C'$ with $\|h\|=\|\phi\|$ and a $\ast$-homomorpshim $\pi: A \to \mathcal{M}(C) \cap \{h\}' \subset B^{\ast \ast}$ such that $\phi(a) = \pi(a)h$ for $a\in A$, where $C$ is the $C^*$-subalgebra in $B$ generated by $\phi(A)$. Given $a,b \in A$, we have
\[
\|\phi(a)\phi(b) \|= \|\pi(a)h \pi(b)h\| = \|h \pi(ab) h\| \leq \|\phi\|^2\cdot \|ab\|,
\]
which concludes the proof.
\end{proof}

%
%

The following notion of nuclear dimension was introduced by Winter and Zacharias:

\begin{defn}[{\cite[Definition 2.1]{WZ10}}]\label{defn:nuclear dimension}
A $C^*$-algebra $A$ is said to have \emph{nuclear dimension at most $n$}, denoted by $\nudim~A \leq n$, if there exists a net $(F_\lambda, \psi_\lambda, \phi_\lambda)_{\lambda \in \Lambda}$ such that the $F_\lambda$ are finite dimensional $C^*$-algebras, and such that $\psi_\lambda: A \to F_\lambda$ and $\phi_\lambda: F_\lambda \to A$ are c.p. maps satisfying:
\begin{enumerate}[(i)]
 \item $\phi_\lambda \circ \psi_\lambda(a) \to a$ uniformly on finite subsets of $A$;
 \item $\|\psi_\lambda\| \leq 1$;
 \item for each $\lambda \in \Lambda$, $F_\lambda$ decomposes into $n+1$ ideals $F_\lambda = F_\lambda^{(0)} \oplus \cdots \oplus F_\lambda^{(n)}$ such that $\phi_{\lambda}^{(i)}:=\phi_\lambda\big|_{F_\lambda^{(i)}}$ is a c.p.c. order zero map for $i=0,1,\cdots,n$.
\end{enumerate}
For convenience, we denote $\psi_{\lambda}^{(i)}: A \to F^{(i)}_{\lambda}$ the $i$-th component of $\psi_{\lambda}$.
\end{defn}

The maps $\psi_{\lambda}^{(i)}$ can be combined into an order zero map as follows, which plays a key role when we define the notion of transfinite nuclear dimension later.

\begin{prop}[{\cite[Proposition 2.2]{Rob11}}]\label{prop:app order zero for psi}
Let $A$ be a $C^*$-algebra with nuclear dimension at most $n$. Then there exist families $\{\psi_{\lambda}^{(i)}\}_{\lambda \in \Lambda}$ and $\{\phi_{\lambda}^{(i)}\}_{\lambda \in \Lambda}$ for $i=0,1,\cdots,n$ satisfying conditions in Definition  \ref{defn:nuclear dimension} and moreover, the induced maps
\[
\psi^{(i)}: A \longrightarrow \prod_\lambda F_\lambda^{(i)}\big/\bigoplus_\lambda F_\lambda^{(i)} \quad \text{and} \quad \phi^{(i)}: \prod_\lambda F_\lambda^{(i)}\big/\bigoplus_\lambda F_\lambda^{(i)} \longrightarrow A_\Lambda:=\prod_\lambda A \big/\bigoplus_\lambda A
\]
are c.p.c. order zero for $i=0,1,\cdots,n$. In this case, we have $\iota = \sum_{i=1}^n \phi^{(i)} \circ \psi^{(i)}$, where $\iota: A \to A_\Lambda$ is induced by the diagonal embedding.
\end{prop}

Recall that the proof of Proposition \ref{prop:app order zero for psi} is a combination of \cite[Proposition 3.2]{WZ10} together with the following lemma (whose proof is contained in that for \cite[Proposition 2.2]{Rob11}):

\begin{lem}\label{lem:lifting lemma}
Let $\phi: C \to D$ be a c.p.c. order zero map between $C^*$-algebras and $\phi|_I=0$ for some closed two-sided ideal $I$ in $C$. Then the induced map $\tilde{\phi}: C/I \to D$ has order zero.
\end{lem}

To define the transfinite extension of the nuclear dimension, let us firstly introduce the following set:

\begin{defn}\label{defn:set for trnudim}
Let $A$ be a $C^*$-algebra, $\F \subseteq A$ be finite and $q \in \NN$. Define $N(\F,q) \subseteq \Fin~\NN$ as follows: $\sigma = \{t_0, t_1, \cdots, t_n\} \in \Fin~\NN$ with $t_0 < t_1 < \cdots < t_n$ belongs to $N(\F,q)$ if and only if $t_0 \geq q$ and there does \emph{not} exist $(F, \psi, \phi)$ where
\begin{enumerate}
 \item[(a)] $F=F^{(0)} \oplus \cdots \oplus F^{(n)}$ is a finite-dimensional $C^*$-algebra;
 \item[(b)] $\psi=\psi^{(0)} \oplus \cdots \oplus \psi^{(n)}: A \to F$ is a c.p.c. map with $\psi^{(i)}:A \to F^{(i)}$ for each $i$;
 \item[(c)] $\phi: F \to A$ is c.p. and the restriction $\phi^{(i)}:=\phi|_{F^{(i)}}: F^{(i)} \to A$ is contractive and order zero for each $i$,
\end{enumerate}
satisfying the following:
\begin{enumerate}
 \item for any $i=0,1,\cdots,n$ and $a,b \in \F$, there exist $a', b' \in A$ with $\|a-a'\| < 1/q$ and $\|b-b'\|<1/q$ such that $\|\psi^{(i)}(a')\psi^{(i)}(b')\| < 1/t_i + \|a'b'\|$;
 \item $\|\phi\circ \psi(a) - a\|<1/q$ for all $a\in \F$.
\end{enumerate}
To avoid ambiguity, sometimes we also denote $N_A(\F,q):=N(\F,q)$.
\end{defn}

For convenience, we call $\sigma = \{t_0, t_1, \cdots, t_n\} \in \Fin~\NN$ \emph{canonical} if $t_0 < t_1 < \cdots < t_n$. For canonical $\sigma = \{t_0,t_1, \cdots, t_n\}$ and $\tau = \{s_0, s_1, \cdots, s_m\}$ in $\Fin~\NN$, denote $\sigma \leq \tau$ if $n = m$ and $t_i \leq s_i$ for $i=0,1,\cdots, n$.

We record the following elementary property for the set $N(\F,q)$, whose proof is straightforward and hence omitted.

\begin{lem}\label{lem:elementary property for N(F,q)}
With the same notation as above, we have:
\begin{enumerate}
 \item $N(\F,q)$ is inclusive, \emph{i.e.}, for any $\sigma \in N(\F,q)$ any $\tau \subseteq \sigma$, then $\tau \in N(\F,q)$.
 \item If $\sigma \in N(\F,q)$ and $\sigma \leq \tau$, then $\tau \in N(\F,q)$.
\end{enumerate}
\end{lem}

The following is our transfinite extension of the nuclear dimension:

\begin{defn}\label{defn:trnudim}
For a $C^*$-algebra $A$ and an ordinal number $\alpha$, we say that the \emph{transfinite nuclear dimension} of $A$ is at most $\alpha$, denoted by $\trnudim~A \leq \alpha$, if for any $q \in \NN$ and any finite subset $\mathcal{F}\subseteq A$, we have $\Ord~N(\F,q) \leq \alpha$. Denote $\trnudim~A = \alpha$ if $\trnudim~A \leq \alpha$ but $\trnudim~A \leq \beta$ fails for any $\beta < \alpha$.
\end{defn}

Now we prove Theorem \ref{introthm:trnudim and nudim}, showing that transfinite nuclear dimension coincides with the nuclear dimension when taking values in natural numbers:

\begin{thm}[Theorem \ref{introthm:trnudim and nudim}]\label{thm:trnudim and nudim}
For a $C^*$-algebra $A$ and $n\in \NN$, we have $\trnudim~A \leq n$ if and only if $\nudim~A \leq n$.
\end{thm}

\begin{proof}
\emph{Necessity}: By assumption, for any $q\in \NN$ and finite $\F \subseteq A$, we have $\Ord~ N(\mathcal{F},q) \leq n$. By Lemma \ref{lem:ordinal finite lemma}, for any $\sigma \in N(\F,q)$ we have $|\sigma| \leq n$, which implies that there exists a canonical $\sigma=\{t_0, t_1, \cdots,t_n\} \notin N(\F,q)$ with $t_0 \geq q$. This concludes $\nudim~A\leq n$.

\emph{Sufficiency}: Now assume $\nudim~A\leq n$. Then Proposition \ref{prop:app order zero for psi} provides families $\{\psi_{\lambda}^{(i)}\}_{\lambda \in \Lambda}$ and $\{\phi_{\lambda}^{(i)}\}_{\lambda \in \Lambda}$ for $i=0,1,\cdots,n$ satisfying the conditions therein. Given $q\in \NN$, finite $\F \subseteq A$ and a canonical $\sigma=\{t_0, t_1, \cdots,t_n\} \in \Fin~\NN$, let us fix an $i \in \{0,1,\cdots,n\}$. Lemma \ref{lem:easy fact for order 0 map} implies that $\|\psi^{(i)}(a)\psi^{(i)}(b)\| \leq \|ab\|$ for any $a,b \in A$. Hence there exists $\lambda_i \in \Lambda$ such that for any $\lambda \geq \lambda_i$, we have $\|\psi^{(i)}_\lambda(a)\psi^{(i)}_\lambda(b)\| < 1/t_i + \|ab\|$ for any $a,b \in \F$. Also choose $\lambda' \in \Lambda$ such that for any $\lambda \geq \lambda'$, we have
\[
\big\| \sum_{i=0}^n \phi^{(i)}_\lambda \circ \psi^{(i)}_\lambda(a) - a \big\| < 1/q \quad \text{for}\quad a\in \F.
\]
Take $\lambda'' \in \Lambda$ greater than each $\lambda_i$ for $i=0,1,\cdots,n$ and $\lambda'$. Then for each $\lambda \geq \lambda''$, we obtain $(F_\lambda, \phi_\lambda, \psi_\lambda)$ satisfying conditions (a)-(c) in Definition \ref{defn:set for trnudim} for the number $n$ such that
\begin{itemize}
 \item $\|\psi^{(i)}_\lambda(a)\psi^{(i)}_\lambda(b)\| < 1/t_i + \|ab\|$ for any $i=0,1,\cdots,n$ and $a,b \in \F$;
 \item $\|\phi_\lambda\circ \psi_\lambda(a) - a\|<1/q$ for all $a\in \F$.
\end{itemize}
If $\trnudim~A \leq n$ does not hold, then there exist $q_0 \in \NN$, finite $\F_0 \subseteq A$ such that $\Ord~N(\F_0,q_0) >n$. Hence by Lemma \ref{lem:ordinal finite lemma}, there exists $\tau \in N(\F_0,q_0)$ such that $|\tau| > n$. Now Lemma \ref{lem:elementary property for N(F,q)} implies that there exists $\sigma \in N(\F_0,q_0)$ with $|\sigma| = n+1$. This leads a contradiction by the analysis above.
\end{proof}

More generally, Robert considered the following notion:

\begin{defn}[{\cite[Definition 3.3]{Rob11}}]\label{defn:nudim at most omega}
A $C^*$-algebra $A$ is said to \emph{have nuclear dimension at most $\omega$} if for each $i \in \NN$ there are nets of c.p.c maps $\psi_{\lambda}^{(i)}:A \to F_{\lambda}^{(i)}$ and $\phi_{\lambda}^{(i)}:F_{\lambda}^{(i)}\to A$ with $F^{(i)}_{\lambda}$ finite dimensional $C^*$-algebras, such that
\begin{enumerate}
  \item the induced maps $\psi^{(i)}: A\to \prod_{\lambda}F_{\lambda}^{(i)}/\bigoplus F_{\lambda}^{(i)}$ and $\phi^{(i)}:\prod_{\lambda}F_{\lambda}^{(i)}/\bigoplus F_{\lambda}^{(i)}\to A_\Lambda $ are c.p.c. and order 0;
  \item $\iota(a) = \sum_{i=0}^{\infty}\phi^{(i)}\psi^{(i)}(a)$ for all $a\in A$.
\end{enumerate}
Here the notation $A_\Lambda$ and $\iota$ are the same as in Proposition \ref{prop:app order zero for psi}.
\end{defn}

To relate our notion of transfinite nuclear dimension to Definition \ref{defn:nudim at most omega}, we consider the following stronger form:

\begin{defn}\label{defn:nudim at most omega stronger form}
A $C^*$-algebra $A$ is said to \emph{strongly have nuclear dimension at most $\omega$} if for each $i \in \NN$ there are nets of c.p.c maps $\psi_{\lambda}^{(i)}:A \to F_{\lambda}^{(i)}$ and $\phi_{\lambda}^{(i)}:F_{\lambda}^{(i)}\to A$ with $F^{(i)}_{\lambda}$ finite dimensional $C^*$-algebras and $\phi_{\lambda}^{(i)}$ order zero, such that
\begin{enumerate}
  \item the induced maps $\psi^{(i)}: A\to \prod_{\lambda}F_{\lambda}^{(i)}/\bigoplus F_{\lambda}^{(i)}$ and $\phi^{(i)}:\prod_{\lambda}F_{\lambda}^{(i)}/\bigoplus F_{\lambda}^{(i)}\to A_\Lambda $ are c.p.c. and order 0;
  \item $\iota(a) = \sum_{i=0}^{\infty}\phi^{(i)}\psi^{(i)}(a)$ for all $a\in A$.
\end{enumerate}
Here the notation $A_\Lambda$ and $\iota$ are the same as in Proposition \ref{prop:app order zero for psi}.
\end{defn}

\begin{rem}
In Definition \ref{defn:nudim at most omega stronger form}, the map $\phi^{(i)}$ automatically has order zero since each $\phi_{\lambda}^{(i)}$ has order zero, which follows from Lemma \ref{lem:lifting lemma}. Here we keep the above form to compare with Definition \ref{defn:nudim at most omega}.

Note that the only difference between the above two definitions is that we require each $\phi_{\lambda}^{(i)}$ has order zero in Definition \ref{defn:nudim at most omega stronger form}, which might not be the case in Definition \ref{defn:nudim at most omega}. Hence it is obvious that strongly having nuclear dimension at most $\omega$ implies having nuclear dimension at most $\omega$. However, it is unclear to the authors whether the converse holds or not.
\end{rem}

Now we aim to show Theorem \ref{introthm:omega nudim and trnudim}, which indicates that our notion of transfinite nuclear dimension has a close relation with Definition \ref{defn:nudim at most omega stronger form}:

\begin{thm}[Theorem \ref{introthm:omega nudim and trnudim}]
For a $C^*$-algebra $A$, we have $\trnudim~A < \infty$ if and only if $A$ strongly has nuclear dimension at most $\omega$.
\end{thm}

To prove Theorem \ref{introthm:omega nudim and trnudim}, we need the following auxiliary result:

\begin{lem}\label{lem:char for fintieness of trnudim}
For a $C^*$-algebra $A$, we have $\trnudim~A < \infty$ \emph{if and only if} for any $q \in \NN$ and finite $\F \subseteq A$, $\Ord~N(\F,q) < \infty$.
\end{lem}

\begin{proof}
The necessity is obvious, and we focus on the sufficiency. By assumption, for any $q \in \NN$ and finite $\F \subseteq A$, denote $\alpha_{(\F,q)}:=\Ord~N(\F,q)$. Take
\[
\alpha:=\sup\{\alpha_{(\F,q)}: \F \text{ is a finite subset of } A \text{ and } q\in \NN\},
\]
which is also an ordinal number. Then we conclude $\trnudim~A \leq \alpha$.
\end{proof}

\begin{proof}[Proof of Theorem \ref{introthm:omega nudim and trnudim}]
\emph{Necessity}: Assume that $\trnudim~A  \leq \alpha$ for some ordinal number $\alpha$. By definition, for any $q\in \NN$ and finite $\F \subseteq A$, we have $\Ord~N(\F,q) \leq \alpha$. Since $N(\F,q)$ is inclusive, $N(\F,q) \neq \Fin (\NN\cap [q,\infty))$ by Lemma \ref{lem:ordinal infinite lemma} and hence we can choose $\sigma_{(\F,q)} \in \Fin (\NN\cap [q,\infty)) \setminus N(\F,q)$. Consider the index set
\[
\Lambda:=\left\{(\F,q): \F \text{ is a finite subset of } A \text{ and } q\in \NN \text{ such that } q\geq \max\{\|a\|^2:a\in \F\}\right\},
\]
where the order is given by $(\F_1, q_1) \leq (\F_2, q_2)$ if and only if $\F_1 \subseteq \F_2$ and $q_1 \leq q_2$. Hence for each $\lambda=(\F,q) \in \Lambda$, there is $\sigma_\lambda=\{t_0, t_1, \cdots, t_{n_\lambda}\} \in \Fin~\NN$ and $(F_\lambda, \psi_\lambda, \phi_\lambda)$ where
\begin{enumerate}
 \item[(a)] $F_\lambda = F^{(0)}_\lambda \oplus \cdots \oplus F^{(n_\lambda)}_\lambda$ is a finite-dimensional $C^*$-algebra;\\[-0.3cm]
 \item[(b)] $\psi_\lambda = \psi^{(0)}_\lambda \oplus \cdots \oplus \psi^{(n_\lambda)}_\lambda: A \to F_\lambda$ is a c.p.c. map with $\psi^{(i)}_\lambda:A \to F^{(i)}_\lambda$ for each $i$;\\[-0.3cm]
 \item[(c)] $\phi_\lambda: F_\lambda \to A$ is c.p. and the restriction $\phi^{(i)}_\lambda:=\phi_\lambda\big|_{F^{(i)}_\lambda}: F^{(i)}_\lambda \to A$ is contractive and order zero for each $i$,
\end{enumerate}
satisfying the following:
\begin{enumerate}
 \item for any $i=0,1,\cdots,n_\lambda$ and $a,b \in \F$ with $a\perp b$, take $a'\in A$ and $b' \in A$ with $\|a-a'\| < 1/q$ and $\|b-b'\| < 1/q$ such that $\|\psi^{(i)}_\lambda(a')\psi^{(i)}_\lambda(b')\| < 1/t_i + \|a'b'\|$. Hence for $M:=\max\{\|a\|: a\in \F\}$, we have
 \begin{eqnarray*}
\|\psi^{(i)}_\lambda(a)\psi^{(i)}_\lambda(b)\|
 &\leq & \|\psi^{(i)}_\lambda(a')\psi^{(i)}_\lambda(b')\| + \|\psi^{(i)}_\lambda(a - a')\psi^{(i)}_\lambda(b')\| + \|\psi^{(i)}_\lambda(a)\psi^{(i)}_\lambda(b-b')\| \\
 & < & \frac{1}{t_i} + \|a'b'\| + \frac{1}{q} \cdot (M+\frac{1}{q}) + M \cdot \frac{1}{q}  \\
 & < & \frac{1}{q} + \frac{2}{q} \cdot (M+\frac{1}{q}) + M \cdot \frac{2}{q} \\
 & \leq & \frac{1}{q} + \frac{2}{q^2} + \frac{4}{q^{1/2}} \to 0 \quad \text{as} \quad q \to \infty.
 \end{eqnarray*}
 \item $\|\phi_\lambda\circ \psi_\lambda(a) - a\|<1/q$ for any $a\in \F$.
\end{enumerate}
These maps can be combined to provide the required maps in Definition \ref{defn:nudim at most omega stronger form} (here $\phi^{(i)}$ having order zero follows from Lemma \ref{lem:lifting lemma}) to ensure that $A$ strongly has nuclear dimension at most $\omega$.

\emph{Sufficiency}: Assume the contrary. Then by Lemma \ref{lem:char for fintieness of trnudim}, there exists $q_0 \in \NN$ and finite $\F_0 \subseteq A$ such that $\Ord~N(\F_0,q_0) = \infty$. Since $N(\F_0,q_0)$ is inclusive by Lemma \ref{lem:elementary property for N(F,q)}, then Lemma \ref{lem:ordinal infinite lemma} implies that there exists a sequence $\{t_k\}_{k\in \NN} \subseteq \NN \cap [q_0,\infty)$ such that $\{t_k\}_{k=0}^n \in N(\F_0,q_0)$ for each $n\in \NN$.

Since $A$ strongly has nuclear dimension at most $\omega$, there exist nets of c.p.c maps $\psi_{\lambda}^{(i)}:A \to F_{\lambda}^{(i)}$ and $\phi_{\lambda}^{(i)}:F_{\lambda}^{(i)}\to A$ satisfying the conditions in Definition \ref{defn:nudim at most omega stronger form}. Hence there exists $N_0\in \NN$ such that
\[
\big\| \iota(a) - \sum_{i=0}^{N_0} \phi^{(i)} \psi^{(i)}(a) \big\| < \frac{1}{q_0} \quad \text{for} \quad a\in \F_0,
\]
which implies that there exists $\lambda_0$ such that for any $\lambda \geq \lambda_0$, we have
\[
\big\| a - \sum_{i=0}^{N_0} \phi^{(i)}_{\lambda} \psi^{(i)}_{\lambda}(a) \big\| < \frac{1}{q_0} \quad \text{for} \quad a\in \F_0.
\]
On the other hand, note that $\psi^{(i)}: A\to \prod_{\lambda}F_{\lambda}^{(i)}/\bigoplus F_{\lambda}^{(i)}$ has order zero for each $i$. Hence using the same argument as in the proof for the sufficiency of Theorem \ref{thm:trnudim and nudim}, there exists $\lambda_1$ such that for any $\lambda \geq \lambda_1$ and $i=0,1,\cdots,N_0$, we have
\[
\left\| \psi^{(i)}_{\lambda}(a) \psi^{(i)}_{\lambda}(b) \right\| < \frac{1}{t_i} + \|ab\| \quad \text{for} \quad a,b \in \F_0.
\]
This implies that $\{t_0, t_1, \cdots,t_{N_0}\} \notin N(\F_0, q_0)$, which is a contradiction. Hence we conclude the proof.
\end{proof}

%

Finally in this section, we study a permanence property to provide an example.

\begin{lem}\label{lem:direct union}
Let $\{A_j\}_{j \in J}$ be an increasing net of $C^*$-algebras (\emph{i.e.}, $A_j \subseteq A_k$ for $j \leq k$), and $A=\overline{\bigcup_j A_j}$ be their direct union. If there exists an ordinal number $\alpha$ such that $\trnudim~A_j \leq \alpha$ for each $j\in J$, then we have $\trnudim~A \leq \alpha$.
\end{lem}

\begin{proof}
Given $q \in \NN$ and finite $\F \subseteq A$, choose $j\in J$ and finite $\F' \subseteq A_j$ such that each element in $\F$ is $\frac{1}{4q}$-close to some element in $\F'$. We claim that the map $\NN \to \NN, n \mapsto 2n$ maps $N_A(\F,q)$ into $N_{A_j}(\F', 2q)$, which will conclude the proof thanks to Lemma \ref{lem:ordinal bijection lemma}.

Assume not, then there exists a canonical $\sigma=\{t_0, t_1, \cdots,t_n\} \in N_A(\F,q)$ while $2\sigma:=\{2t_0, 2t_1, \cdots, 2t_n\} \notin N_{A_j}(\F', 2q)$. By definition, there exists $(F, \psi_j, \phi_j)$ where
\begin{enumerate}
 \item[(a)] $F=F^{(0)} \oplus \cdots \oplus F^{(n)}$ is a finite-dimensional $C^*$-algebra;
 \item[(b)] $\psi_j=\psi^{(0)}_j \oplus \cdots \oplus \psi^{(n)}_j: A_j \to F$ is a c.p.c. map with $\psi^{(i)}_j:A_j \to F^{(i)}$ for each $i$;
 \item[(c)] $\phi_j: F \to A_j$ is c.p. and the restriction $\phi^{(i)}_j:=\phi_j\big|_{F^{(i)}}: F^{(i)} \to A_j$ is contractive and order zero for each $i$,
\end{enumerate}
satisfying the following:
\begin{enumerate}
 \item for any $i=0,1,\cdots,n$ and $a',b' \in \F'$, there exist $a'', b'' \in A_j$ with $\|a'-a''\| < 1/2q$ and $\|b'-b''\|<1/2q$ such that $\|\psi^{(i)}_j(a'')\psi^{(i)}_j(b'')\| < 1/2t_i + \|a''b''\|$;
 \item $\|\phi_j\circ \psi_j(a') - a'\|<1/2q$ for all $a'\in \F'$.
\end{enumerate}

By Arveson's extension theorem, we can extend $\psi^{(i)}_j$ to a c.p.c. map $\psi^{(i)}:A \to F^{(i)}$ for each $i$. Then $\psi:=\psi^{(0)} \oplus \cdots \oplus \psi^{(n)}: A \to F$ is also a c.p.c. map. Denote $\phi: F \to A$ the composition of $\phi_j$ with the inclusion map $A_j \hookrightarrow A$, which is c.p., and the restriction $\phi^{(i)}:=\phi\big|_{F^{(i)}}: F^{(i)} \to A$ is contractive and order zero for each $i$.

For any $i=0,1,\cdots,n$ and $a,b \in \F$, choose $a', b' \in \F'$ with $\|a-a'\| < 1/4q$ and $\|b-b'\|<1/4q$. Then there exist $a'', b'' \in A_j$ with $\|a'-a''\| < 1/2q$ and $\|b'-b''\|<1/2q$ such that $\|\psi^{(i)}_j(a'')\psi^{(i)}_j(b'')\| < 1/t_i + \|a''b''\|$. Hence $\|a-a''\| < 1/q$, $\|b-b''\|<1/q$ and $\|\psi^{(i)}(a'')\psi^{(i)}(b'')\| < 1/2t_i + \|a''b''\|$ since $\psi^{(i)}$ extends $\psi^{(i)}_j$. Moreover, for any $a\in \F$ choose $a' \in \F'$ with $\|a-a'\| < 1/4q$ and hence:
\[
\|\phi\circ \psi(a) - a\| < \|\phi\circ \psi(a') - a'\| + 2 \cdot \frac{1}{4q} = \|\phi_j\circ \psi_j(a') - a'\| + \frac{1}{2q} < \frac{1}{2q} + \frac{1}{2q} = \frac{1}{q}.
\]
This shows that $\sigma \notin N_A(\F,q)$, which leads to a contradiction.
\end{proof}

\begin{cor}\label{cor:inductive limit for fintie nudim}
Let $\{A_j\}_{j\in J}$ be an inductive system of $C^*$-algebras where each $A_j$ has finite nuclear dimension, and $A$ be their inductive limit. Then $\trnudim~A \leq \omega$, where $\omega$ is the first infinite ordinal number.
\end{cor}

\begin{proof}
Assume that $\theta_j: A_j \to A$ is the associated morphism. Then $\{\theta_j(A_j)\}_{j\in J}$ is an increasing net of $C^*$-subalgebras in $A$ and $A$ is their direct union. By \cite[Proposition 2.3(iv)]{WZ10}, each $\theta_j(A_j)$ has finite nuclear dimension and hence by Theorem \ref{introthm:trnudim and nudim}, $\trnudim~\theta_j(A_j) < \omega$. Therefore, the result follows from Lemma \ref{lem:direct union}.
\end{proof}

The following example follows directly from Corollary \ref{cor:inductive limit for fintie nudim}:

\begin{ex}
Let $X$ be a compact Hausdorff space with finite covering dimension. Then any Villadsen algebra of the first type (see \cite[Definition 3.2]{TW09})
\[
A = \lim_{i\to \infty} (\M_{m_i}(\CC) \otimes C(X^{n_i}), \phi_i)
\]
has transfinite nuclear dimension at most $\omega$. Similarly, any Villadsen algebra of the second type (see \cite{Vil98}) has transfinite nuclear dimension at most $\omega$.
\end{ex}

\section{Transfinite asymptotic dimension and uniform Roe algebra}\label{sec:unif. Roe}

In this section, we would like to study the transfinite nuclear dimension for uniform Roe algebras. Our main task here is to prove Theorem \ref{introthm:trnudim and trasdim}.


Given a discrete metric space $(X,d)$ of bounded geometry, each operator $T \in \B(\ell^2(X))$ can be written in the matrix form $T=(T_{x,y})_{x,y\in X}$, where $T_{x,y}=\langle T \delta_y, \delta_x \rangle \in \CC$. Given an operator $T \in \B(\ell^2(X))$, we define the \emph{propagation} of $T$ to be
\[
\ppg(T):= \sup\{d(x,y): (x,y) \in X \times X \text{ such that } T_{x,y} \neq 0\}.
\]
We say that $T$ has \emph{finite propagation} if $\ppg(T) < \infty$.

The set of all finite propagation operators in $\B(\ell^2(X))$ forms a $\ast$-algebra, called the \emph{algebraic uniform Roe algebra of $(X, d)$} and denoted by $\CC_u[X]$. The \emph{uniform Roe algebra of $(X, d)$} is defined to be the operator norm closure of $\CC_u[X]$ in $\B(\ell^2(X))$, which forms a $C^*$-algebra and is denoted by $C^*_u(X)$.

In \cite{WZ10}, Winter and Zacharias showed that the nuclear dimension of the uniform Roe algebra $C^*_u(X)$ is bounded by the asymptotic dimension for $(X,d)$ introduced by Gromov in \cite{Gro93}. Recall that a metric space $(X,d)$ is said to have \emph{asymptotic dimension at most $n$} for some $n\in \NN$, denoted by $\asdim~X \leq n$, if for every $r>0$ there exist uniformly bounded families $\U_0, \U_1, \cdots, \U_n$ of subsets of $X$ such that $\bigcup_{i=0}^n \U_i$ covers $X$ and each $\U_i$ is $r$-disjoint for $i=0,1,\cdots,n$.

\begin{prop}[{\cite[Theorem 8.5]{WZ10}}]\label{prop:WZ09 uniform Roe}
Let $(X,d)$ be a discrete metric space of bounded geometry, then $\nudim~C^*_u(X) \leq \asdim~X$.
\end{prop}

The notion of asymptotic dimension was generalised to the transfinite case:

\begin{defn}[\cite{Rad10}]\label{defn:trasdim}
Given a metric space $(X,d)$, we define a set $A(X,d) \subseteq \Fin~\NN$ as follows: $\sigma = \{t_0, t_1, \cdots, t_n\} \in \Fin~\NN$ belongs to $A(X,d)$ if and only if there are no uniformly bounded families $\U_0, \U_1, \cdots,\U_n$ such that $\bigcup_{i=0}^n \U_i$ covers $X$ and each $\U_i$ is $t_i$-disjoint for $i=0,1,\cdots,n$. The \emph{transfinite asymptotic dimension} of $(X,d)$ is defined to be $\trasdim~X = \Ord~A(X,d)$.
\end{defn}

It was noted in \cite{Rad10} that for $n\in \NN$, $\asdim~X \leq n$ if and only if $\trasdim~X \leq n$. Moreover, it is clear that $A(X,d)$ is inclusive and hence according to Lemma \ref{lem:ordinal infinite lemma}, we have:

\begin{prop}[{\cite[Proposition 1]{Rad10}}]\label{prop:Property C}
For a metric space $(X,d)$, we have $\trasdim~X < \infty$ if and only if $X$ has asymptotic property C.
\end{prop}

Here the notion of asymptotic property C was introduced by Dranishnikov \cite{Dra00}. A metric space $(X,d)$ is said to have \emph{asymptotic property C} if for any sequence of natural numbers $n_1 < n_2 < \cdots$, there exists a finite sequence of uniformly bounded families $\U_0, \U_1, \cdots,\U_n$ such that $\bigcup_{i=0}^n \U_i$ covers $X$ and $\U_i$ is $n_i$-disjoint.

Inspired by Proposition \ref{prop:WZ09 uniform Roe}, we study whether the transfinite nuclear dimension of the uniform Roe algebra can be bounded by the transfinite asymptotic dimension of the underlying space:

\begin{thm}[Theorem \ref{introthm:trnudim and trasdim}]
Let $(X,d)$ be a discrete metric space of bounded geometry. Then we have $\trnudim~C^*_u(X) \leq \trasdim~X$.
\end{thm}

We need the following well-known auxiliary lemma:

\begin{lem}\label{lem:estimate for commutant}
Let $(X,d)$ be a discrete metric space of bounded geometry. For each $R>0$, denote $N_R:=\sup_{x\in X} |B(x,R)|$. Then for any $a \in \CC_u[X]$ and $\alpha$-H\"{o}lder function $h\in \ell^\infty(X)$ with constant $C$, we have
\[
\|[a,h]\| \leq C \cdot \|a\|  \cdot \ppg(a)^\alpha \cdot N_{\ppg(a)}.
\]
\end{lem}

\begin{proof}
For $x,y\in X$, we have $[a,h]_{x,y} = a_{x,y} \cdot (h(x) - h(y))$. Hence we have $\left|[a,h]_{x,y}\right| \leq \left|a_{x,y}\right|\cdot C \cdot \ppg(a)^\alpha$. Now the result follows directly from \cite[Lemma 8.1]{WZ10}.
\end{proof}

We also need the following result to deal with products of matrix algebras. For $n\in \NN$, denote $\M_n(\CC)$ the complex valued $n \times n$-matrix algebra.

\begin{lem}[{\cite[Lemma 8.4]{WZ10}}]\label{lem:prod of matrix}
Let $K$ be an index set and $(n_k)_{k\in K}$ be a bounded family of positive numbers. Then for any $\varepsilon>0$ and finite $\F \subseteq \prod_{k\in K} \M_{n_k}(\CC)$, there is a finite dimensional $C^*$-algebra $F$ and $\ast$-homomorphisms $\psi: \prod_{k\in K} \M_{n_k}(\CC) \to F$ and $\phi: F \to \prod_{k\in K} \M_{n_k}(\CC)$ such that $\|\phi \circ \psi(a) - a\| < \varepsilon$ for each $a\in \F$.
\end{lem}

\begin{proof}[Proof of Theorem \ref{introthm:trnudim and trasdim}]
Given $q \in \NN$ and a finite subset $\F=\{b_0, b_1, \cdots,b_m\} \subseteq C^*_u(X)$, we would like to show $\Ord~N(\F,q) \leq \Ord~A(X,d)$. Set
\[
K:=2\cdot \max\left\{\|b_j\|: j=0,1,\cdots,m\right\}+1.
\]
Since $\CC_u[X]$ is dense in $C^*_u(X)$, we can choose for each $j=0,1,\cdots,m$ an element $a_j \in \CC_u[X]$ such that $\|a_j - b_j\| \leq \frac{1}{12qK}$. Then we have $\|a_j\| \leq K$ for each $j$. Take $M$ to be an integer greater than:
\[
\sup\left\{\|a_j\|  \cdot \ppg(a_j)^{1/2} \cdot N_{\ppg(a_j)}: j=0,1,\cdots,m\right\}.
\]

Choose an injective and monotonely increasing map $\Theta: \NN \to \NN$ such that $\Theta(t) > 12t^2M^2K^2$ for $t\in \NN$, and $\Theta(\NN) \subseteq \{3\cdot (3qM)^2 \cdot 2^{2k}: k\in \NN \setminus \{0\}\}$. According to Lemma \ref{lem:ordinal bijection lemma}, it suffices to show that $\Theta(\sigma) \in A(X,d)$ whenever $\sigma \in N(\F,q)$. Fix a canonical $\sigma= \{t_0, t_1, \cdots, t_n\} \in N(\F,q)$, and assume that $\Theta(\sigma) \notin A(X,d)$. For each $i=0,1,\cdots,n$, denote $s_i:=\Theta(t_i)$ and then $\Theta(\sigma) = \{s_0, s_1, \cdots, s_n\}$. By assumption, there exist uniformly bounded families $\U_0, \U_1, \cdots,\U_n$ such that $\bigcup_{i=0}^n \U_i$ covers $X$ and each $\U_i$ is $s_i$-disjoint for $i=0,1,\cdots,n$. Since $\Theta(\NN) \subseteq \{3\cdot (3qM)^2 \cdot 2^{2k}: k\in \NN \setminus \{0\}\}$, it follows from Lemma \ref{lem:cover for C(q)} that there exist non-negative functions $f_0, f_1, \cdots,f_n$ on $X$ satisfying the following:
\begin{enumerate}
  \item $\supp(f_i) \subset \Nd_{s_i/2}(\bigcup \U_i)$ for $i=0,1,\cdots,n$;\\[-.3cm]
  \item $\sum_{i=0}^n f_i^2 \equiv 1$; \\[-.3cm]
  \item each $f_i$ is $1/2$-H\"{o}lder with constant at most $\sqrt{\frac{3}{s_i}}$, and we have $\sum\limits_{i=0}^n \sqrt{\frac{3}{s_i}} < \frac{1}{3qM}$.
\end{enumerate}

Following the proof for \cite[Theorem 8.5]{WZ09}, denote
\[
A^{(i)}:=\prod_{U \in \U_i} \M_{\Nd_{s_i/2}(U)}(\CC) \quad \text{for} \quad i=0,1,\cdots, n.
\]
We construct the following maps:
\[
\Psi: C^*_u(X) \longrightarrow A^{(0)} \oplus A^{(1)} \oplus \cdots \oplus A^{(n)} \quad \text{by} \quad a \mapsto (f_0af_0, f_1af_1, \cdots,f_naf_n),
\]
and
\[
\Phi: A^{(0)} \oplus A^{(1)} \oplus \cdots \oplus A^{(n)} \longrightarrow C^*_u(X) \quad \text{by} \quad (a'_0, a'_1, \cdots, a'_n) \mapsto a'_0 + a'_1 + \cdots + a'_n.
\]
Here we regard each $A^{(i)}$ as a $C^*$-subalgebra in $C^*_u(X)$ thanks to the choice above.
Obviously $\Psi$ is c.p.c., $\Phi$ is c.p., and $\Phi \circ \Psi(\Id) = \sum_{i=0}^n f_i^2 = \Id$ (which implies that $\|\Phi \circ \Psi\| \leq 1$). Denote the associated maps $\Psi^{(i)}: C^*_u(X) \to A^{(i)}$ and $\Phi^{(i)}: A^{(i)} \to C^*_u(X)$. It is clear that each $\Phi^{(i)}$ is a $\ast$-homomorphism, and hence has order zero.

Now we estimate $\|\Phi \circ \Psi (b) - b\|$ for $b\in \F$. For $j=0,1,\cdots, m$, we have:
\begin{eqnarray*}
\|\Phi \circ \Psi (b_j) - b_j\| &\leq & \|\Phi \circ \Psi (a_j) - a_j\| + \frac{1}{6qK} = \left\| \sum_{i=0}^n f_i a_j f_i - \sum_{i=0}^n f_i^2 a_j\right\| + \frac{1}{6qK} \\
&=& \big\| \sum_{i=0}^n f_i [a_j, f_i]\big\| + \frac{1}{6qK}  \leq  \sum_{i=0}^n \left\|[a_j, f_i]\right\| + \frac{1}{6qK}  \\
& \leq & \sum_{i=0}^n M \cdot \sqrt{\frac{3}{s_i}} + \frac{1}{6qK} < \frac{1}{3q} + \frac{1}{6qK} < \frac{1}{2q},
\end{eqnarray*}
where we use Lemma \ref{lem:estimate for commutant} for the third to last inequality.
On the other hand, for each $i=0,1,\cdots,n$ and $b_j,b_k \in \F$, we have $\|b_j-a_j\| < \frac{1}{12qK}$ and $\|b_k-a_k\| < \frac{1}{12qK}$. Hence
\begin{eqnarray*}
\left\|\Psi^{(i)}(a_j)\Psi^{(i)}(a_k)\right\| &=& \left\| f_ia_jf_i \cdot f_ia_kf_i \right\|\\
&\leq & \left\| f_i^2 a_j \cdot a_kf_i^2 \right\| + \left\| f_i^2a_j \cdot [f_i,a_k]f_i \right\| + \left\| f_i[f_i,a_j] \cdot f_ia_kf_i \right\|\\
& \leq & \|a_ja_k\| + \left\|[f_i,a_k]\right\| \cdot \|a_j\| + \left\| [f_i,a_j] \right\| \cdot \|a_k\|\\
& \leq & \|a_ja_k\| + \left\|[f_i,a_k]\right\| \cdot K + \left\|[f_i,a_j]\right\| \cdot K \\
& \leq & \|a_ja_k\| + 2\sqrt{\frac{3}{s_i}} \cdot MK\\
& <& \|a_ja_k\| + \frac{1}{t_i}
\end{eqnarray*}
where we use Lemma \ref{lem:estimate for commutant} in the penultimate inequality and the last one follows from the choice of $\Theta$.

Finally, we apply Lemma \ref{lem:prod of matrix} to each $A^{(i)}$ and obtain $\ast$-homomorphisms $\tilde{\psi}^{(i)} : A^{(i)} \to F^{(i)}$ and $\tilde{\phi}^{(i)}: F^{(i)} \to A^{(i)}$ for some finite-dimensional $C^*$-algebra $F^{(i)}$ such that $\left\| \tilde{\phi}^{(i)} \circ \tilde{\psi}^{(i)}(a') - a' \right\| < \frac{1}{2q(n+1)}$ for any $a' \in \Psi^{(i)}(\F)$. Setting
\[
\psi^{(i)}:=\tilde{\psi}^{(i)} \circ \Psi^{(i)}: C^*_u(X) \longrightarrow F^{(i)} \quad \text{and} \quad \phi^{(i)}:=\Phi^{(i)} \circ \tilde{\phi}^{(i)}:  F^{(i)} \longrightarrow C^*_u(X),
\]
and $F:=F^{(0)} \oplus \cdots \oplus F^{(n)}$, we obtain c.p. maps
\[
\psi:=(\psi^{(i)})_{i=0}^n: C^*_u(X) \longrightarrow F \quad \text{and} \quad \phi:=\sum_{i=0}^n \phi^{(i)}: F \longrightarrow C^*_u(X).
\]
Now we conclude:
\begin{enumerate}[$\bullet$]
  \item $\psi$ is c.p.c, and each $\phi^{(i)}$ is c.p.c. and order zero;\\[-.3cm]
  \item for any $i=0,1,\cdots,n$ and $a,b \in \F$, there exist $a', b' \in A$ with $\|a-a'\| < 1/q$ and $\|b-b'\| < 1/q$ such that
  \[
  \|\psi^{(i)}(a')\psi^{(i)}(b')\| = \big\| \tilde{\psi}^{(i)} \Psi^{(i)}(a') \cdot \tilde{\psi}^{(i)} \Psi^{(i)}(b') \big\| = \big\| \tilde{\psi}^{(i)}\left(\Psi^{(i)}(a') \cdot\Psi^{(i)}(b')\right) \big\| < \|a'b'\| + 1/t_i.
  \]
  \item for any $a\in \F$, we have
  \[
  \|\phi\circ \psi(a) - a\| = \left\| \sum_{i=0}^n \Phi^{(i)}\tilde{\phi}^{(i)}\tilde{\psi}^{(i)} \Psi^{(i)}(a) - a \right\| \leq \left\| \sum_{i=0}^n \Phi^{(i)}\Psi^{(i)}(a) - a \right\| + \frac{1}{2q} < \frac{1}{q}.
  \]
\end{enumerate}
This means that $\sigma \notin N(\F,q)$, which is a contradiction. Therefore, we conclude the proof.
\end{proof}

\begin{ex}\label{ex:wreath product}
The first named author and Y. Wu proved $\trasdim(\mathbb{Z} \wr \mathbb{Z})\leq \omega+1$ in \cite{ZW23}. Hence applying Theorem \ref{introthm:trnudim and trasdim}, we obtain $\trnudim~ C_u^*(\mathbb{Z}\wr\mathbb{Z})\leq \omega+1$. Together with Radul, they also proved in \cite{WZR22} that given an infinitely countable ordinal number $\alpha$, there exists a metric space $X_{\alpha}$ with $\trasdim~X_{\alpha}=\alpha$. Hence applying Theorem \ref{introthm:trnudim and trasdim} again, we obtain $\trnudim~C_u^*(X_{\alpha}) \leq\alpha$.
\end{ex}

Combining Corollary \ref{introcor:cor to thm omega}, Proposition \ref{prop:Property C} and Theorem \ref{introthm:trnudim and trasdim}, we finally obtain Corollary \ref{introcor:cor to thm asdim}.


\section{Some discussion on the commutative case}\label{sec:commu. case}

In this section, we would like to study the transfinite nuclear dimension for unital commutative $C^*$-algebras and relate them to the transfinite topological covering dimension of their spectra. For simplicity, here we only focus on the case that the spectrum is metrisable.

For a compact metric space $X$, recall that (see, \emph{e.g.}, \cite[Definition 1.6.7]{Eng78}) its \emph{covering dimension} is no more than $n$, denoted by $\dim X \leq n$, if for every finite open cover $\U$ of $X$ there exists a finite open refinement $\V$ such that every point in $X$ belongs to at most $n+1$ elements in $\V$.

For a family $\U$ of subsets in $X$, we say that it is \emph{$C$-bounded} for some $C>0$ if the diameter of each element in $\U$ is bounded by $C$, and \emph{disjoint} if $U \cap V = \emptyset$ for $U,V \in \U$ with $U \neq V$.



We recall the following characterisation, which is a combination of \cite[Theorem 1.6.12]{Eng78} and \cite[Proposition 1.5]{KW04}:

\begin{prop}\label{thm:strongequconditionofdim}
For a compact metric space $(X,d)$, we have $\dim X\leq n$ \emph{if and only if} for any $\varepsilon>0$, there is an $\varepsilon$-bounded finite open cover $\V$ of $X$ which can be decomposed as $\V=\V_0 \cup \cdots \cup \V_n$ where each $\V_i$ is disjoint for $i=0,1,\cdots, n$.
\end{prop}

We introduce a version of transfinite covering dimension using coverings as follows:


\begin{defn}\label{s-trdim}
For a compact metric space $(X,d)$, define $M_{D(X)} \subseteq \Fin~\NN$ as follows: $\sigma=\{t_0,\cdots,t_n\} \in M_{D(X)}$ if and only if there does not exist an open cover $\U$ which can be decomposed as $\U=\U_0 \cup \cdots \cup \U_n$ such that $\U_i$ is disjoint and $1/t_i$-bounded for each $i=0,1,\cdots,n$. The \emph{transfinite covering dimension} of $X$, denoted by ${\rm d}$-$\trdim X$, is defined to be $\Ord~M_{D(X)}$.
\end{defn}

Combining Lemma \ref{lem:ordinal finite lemma} with Proposition \ref{thm:strongequconditionofdim}, we have:

\begin{lem}\label{s-trdim=dim}
For a compact metric space $X$ and $n\in \NN$, ${\rm d}$-$\trdim X \leq n $ \emph{if and only if} $\dim X \leq n$.
\end{lem}

\begin{rem}
Let us explain the notation ``${\rm d}$-$\trdim X$''. In \cite{Bor88}, Borst introduced another transfinite extension for covering dimension using partitions. Recall that for a compact metric space $X$ and disjoint closed subsets $A,B \subseteq X$, a \emph{partition between $A$ and $B$} is a subset $L \subseteq X$ such that there exist open subsets $U$ and $V\subseteq X$ satisfying $A \subseteq U$, $B \subseteq V$, $U \cap V = \emptyset$ and $L=X \setminus (U \cup V)$. Denote $L(X)$ the set of all pairs $(A,B)$ of disjoint closed subsets of $X$ and by $M_{L(X)}$ the class of all finite pairs $\{(A_i,B_i)\}_{i=0}^n$ of $L(X)$ such that if $L_i$ is a partition between $A_i$ and $B_i$, then $\bigcap_{i=0}^nL_i\neq\emptyset$. Borst defined his \emph{transfinite covering dimension} of $X$ to be $\Ord~M_{L(X)}$.

It is unclear to the authors whether Borst's definition is compatible with ours and to tell the (possible) difference, we choose the notation ``${\rm d}$-$\trdim X$'' to emphasise that our definition makes use of diameters of elements in covers. It turns out that our definition is closely related to certain version of transfinite nuclear dimension as follows, which is our motivation to introduce Definition \ref{s-trdim}.
\end{rem}

\begin{defn}\label{defn:set for s-trnudim}
Let $X$ be a compact metric space, $\F \subseteq C(X)$ be finite and $q \in \NN$. Define $M(\F,q) \subseteq \Fin~\NN$ as follows: $\sigma = \{t_0, t_1, \cdots, t_n\} \in \Fin~\NN$ with $t_0 < t_1 < \cdots < t_n$ belongs to $M(\F,q)$ if and only if $t_0 \geq q$ and there does \emph{not} exist $(F, \psi, \phi)$ where
\begin{enumerate}
 \item[(a)] $F=F^{(0)} \oplus \cdots \oplus F^{(n)}$ and $F^{(i)}=\bigoplus_{j=1}^{n_i}\mathbb{C}$ for each $i=0,\cdots,n$;
 \item[(b)] $\psi=\psi^{(0)} \oplus \cdots \oplus \psi^{(n)}: C(X) \to F$ is a c.p.c. map with $\psi^{(i)}:C(X) \to F^{(i)}$ for each $i$;
 \item[(c)] $\phi: F \to C(X)$ is c.p. and the restriction $\phi^{(i)}:=\phi|_{F^{(i)}}: F^{(i)} \to C(X)$ is contractive and order zero for each $i$,
\end{enumerate}
satisfying the following:
\begin{enumerate}
 \item for any $i=0,1,\cdots,n$ and $a,b \in \F$, there exist $a', b' \in C(X)$ with $\|a-a'\| < 1/q$ and $\|b-b'\|<1/q$ such that $\|\psi^{(i)}(a')\psi^{(i)}(b')\| < 1/t_i + \|a'b'\|$;
 \item $\|\phi\circ \psi(a) - a\|<1/q$ for all $a\in \F$;
 \item for any $ i=0,\cdots,n$ and minimal projection $e \in F^{(i)}$, $\diam (\supp ~\phi^{(i)}(e) )<1/ t_i$.
\end{enumerate}
\end{defn}

\begin{rem}\label{diffbetweentrdimands-trdim}
Note that condition (3) above is only designed for $C(X)$ where $X$ is a metric space. This is the only difference between  $N(\mathcal{F},q)$ defined in Definition \ref{defn:set for trnudim} and $M(\mathcal{F},q)$ above.
\end{rem}

\begin{defn}\label{defn:strong trnudim}
For a compact metric space $X$ and an ordinal number $\alpha$, we say that the \emph{strong transfinite nuclear dimension} of $C(X)$ is at most $\alpha$, denoted by ${\rm s}$-$\trnudim~A \leq \alpha$, if for any $q \in \NN$ and any finite subset $\mathcal{F}\subseteq A$, we have $\Ord~M(\F,q) \leq \alpha$. Denote ${\rm s}$-$\trnudim~C(X) = \alpha$ if ${\rm s}$-$\trnudim~C(X) \leq \alpha$ but ${\rm s}$-$\trnudim~C(X) \leq \beta$ fails for any $\beta < \alpha$.
\end{defn}

The following is the main result of this section, which also explains the terminology.

\begin{prop}\label{relaforanyordinal}
For a compact metric space $X$, we have:
\[
\trnudim~C(X) \leq {\rm s}\text{-}\trnudim~C(X) = {\rm d}\text{-}\trdim~X.
\]
\end{prop}

\begin{proof}
By Remark \ref{diffbetweentrdimands-trdim}, it is clear that $\trnudim~C(X) \leq {\rm s}\text{-}\trnudim~C(X)$. Hence we only focus on the second equality.

To see ${\rm s}\text{-}\trnudim~C(X) \geq {\rm d}\text{-}\trdim~X$, we claim that $M_{D(X)}\subset M(\{1_X\},1)$. If not, there exists $(t_0,t_1,\cdots,t_n)\in M_{D(X)} \setminus M(\{1_X\},1)$. Hence there exists $(F,\psi,\phi)$ satisfying the conditions in Definition \ref{defn:set for s-trnudim} (with the same notation) and $\|\phi\circ \psi(1_X) - 1_X\|<1$, where $1_X$ is the constant function $1$ on $X$. For any $i=0,\cdots,n$ and $j=1,\cdots,n_i$, let
\[
U_j^{(i)} = \left\{x\in X:\phi^{(i)}(e_j^{(i)})(x)>0\right\} \quad \text{and} \quad \mathcal{U}_i=\{U_1^{(i)},\cdots,U_{n_i}^{(i)}\},
\]
where $e_j^{(i)}$ is the $j$-th minimal projection in $F^{(i)}=\bigoplus_{j=1}^{n_i}\CC$.
By assumption, each $\mathcal{U}_i$ is disjoint and $\frac{1}{t_i}$-bounded.
Moreover, $\|\phi\circ \psi(1_X) - 1_X\|<1$ implies that $\bigcup_{i=0}^n\U_i$ is a cover of $X$ and hence $(t_0,t_1,\cdots,t_n)\notin M_{D(X)}$. This leads to a contradiction.

For the converse, fixing a finite subset $\F \subseteq C(X)$ and $q\in \NN$, we need to show $\Ord~M(\F,q) \leq {\rm d}\text{-}\trdim~X$. Since $X$ is compact, there exists a finite open cover $\U$ of $X$ such that $|a(x) - a(y)| < \frac{1}{q}$ for any $a\in \F$ and $x,y\in U$ for $U \in \U$. Take $M \in \NN$ such that $\frac{1}{M}$ is less than the Lebesgue number $L$ of $\U$. Define a map $\Phi: \NN \to \NN$ by $n \mapsto n+M$.

We claim that $\Phi$ maps $M(\F,q)$ to $M_{D(X)}$. If not, there exists $\sigma=\{t_0, t_1, \cdots, t_n\} \in M(\F,q)$ such that $\Phi(\sigma)$ does not belong to $M_{D(X)}$. Hence there exists an open cover $\V$ which can be decomposed as $\V=\V_0 \cup \cdots \cup \V_n$ such that $\V_i$ is disjoint and $\frac{1}{t_i+M}$-bounded for each $i=0,1,\cdots,n$. Since $\frac{1}{t_i+M} < \frac{1}{M} < L$, then $\V$ is an open refinement of $\U$. Eliminating elements in $\V$ if necessary, we can further assume that for each $V \in \V$, there exists $x_V \in V$ while $x_V \notin V'$ for any $V' \neq V$ in $\V$. Taking a partition of unity $\{h_V: V \in \V\}$ subordinate to $\V$, then $h_V(x_V) = 1$ for $V \in \V$. For $i=0,1,\cdots, n$, denote $\V_i:=\{V^{(i)}_1, \cdots,V^{(i)}_{n_i}\}$ and $x^{(i)}_j:=x_{V^{(i)}_j}$, and we define:
\begin{itemize}
 \item $F^{(i)}=\bigoplus_{j=1}^{n_i}\mathbb{C}$;
 \item $\psi^{(i)}:C(X) \to F^{(i)}$ by $\psi^{(i)}(a):=(a(x^{(i)}_1), \cdots, a(x^{(i)}_{n_i}))$;
 \item $\phi^{(i)}: F^{(i)} \to C(X)$ by $\phi^{(i)}(e^{(i)}_j)=h_{V^{(i)}_j}$, where $e^{(i)}_j$ is the $j$-th minimal projection in $F^{(i)}=\bigoplus_{j=1}^{n_i}\mathbb{C}$.
\end{itemize}
Then $\psi=\psi^{(0)} \oplus \cdots \oplus \psi^{(n)}: C(X) \to F=F^{(0)} \oplus \cdots \oplus F^{(n)}$ is a c.p.c. map, the map $\phi=(\phi^{(i)})_{i=0}^n: F \to C(X)$ is c.p. and each $\phi^{(i)}$ is contractive and order zero. Moreover, we have:
\[
\diam \left(\supp ~\phi^{(i)}(e_j^{(i)})\right) = \diam \left(\supp~h_{V^{(i)}_j}\right) \leq \diam \left(V^{(i)}_j\right) \leq \frac{1}{t_i + M} < \frac{1}{t_i}.
\]
Also it is clear that $\|\psi^{(i)}(a)\psi^{(i)}(b)\| \leq \|ab\|$ for each $i$ and any $a,b \in C(X)$ since $\psi^{(i)}$ is a $\ast$-homomorphism. Finally, note that for each $a \in \F$ and $x\in X$ we have:
\[
|\phi\circ \psi(a)(x) - a(x)|  = \big| \sum_{V \in \V} a(x_V) h_V(x) - \sum_{V \in \V} a_V(x) h_V(x) \big| \leq \sum_{V \in \V} \left| a(x_V) - a(x) \right| \cdot h_V(x) < \frac{1}{q}.
\]
Hence we conclude that $\sigma=\{t_0, t_1, \cdots, t_n\} \notin M(\F,q)$, which is a contradiction. Finally applying Lemma \ref{lem:ordinal bijection lemma}, we finish the proof.
\end{proof}

Combining Theorem \ref{thm:trnudim and nudim}, Lemma \ref{s-trdim=dim}, Proposition \ref{relaforanyordinal} and \cite[Proposition 2.4]{WZ10}, we reach the following:

\begin{cor}\label{cor:relaforfinite}
Let $X$ be a compact metric space. Then the following holds
\[
\nudim~C(X) = \trnudim~C(X) = {\rm s}\text{-}\trnudim~C(X) = {\rm d}\text{-}\trdim~X = \dim X
\]
if any of them is finite.
\end{cor}

Finally, we pose the following question which is currently unclear to the authors:
\begin{question}
For a compact metric space $X$, do we always have $\trnudim~C(X) = {\rm s}\text{-}\trnudim~C(X)$?
\end{question}

%

\bibliographystyle{plain}
\bibliography{bib_Trnudim}

\end{document}